\documentclass[a4paper,11pt,leqno]{article}
\author{F\'elix del Teso M\'endez}
\usepackage[usenames,dvipsnames,svgnames,table]{xcolor}
\usepackage{amsmath,amsthm,amssymb,verbatim,tabularx,graphicx}
\usepackage{fancyhdr}
\usepackage{enumerate}
\usepackage{tocbibind}
\usepackage{epic}
\usepackage{pgf,tikz}
\usetikzlibrary{arrows}
\newtheorem{teor}{Theorem}[section]

\newtheorem{nota}{Note}

\newtheorem{lemma}[teor]{Lemma}
\newtheorem{coro}[teor]{Corollary}

\newcommand\RR{\mathbb{R}^N}

\newcommand\lap{(-\Delta)^{\sigma/2}}

\newcommand\musig{\mu_\sigma}
\newcommand\nusig{\nu_\sigma}
\numberwithin{equation}{section}

\begin{document}

\begin{center}
{\LARGE\textbf{Finite difference method for a general \\[6pt] fractional porous medium equation}}

\vspace{.5cm}

\Large{ by \ F\'elix del Teso and  Juan Luis V\'azquez\\[6pt]
\em Universidad Aut\'onoma de Madrid \footnote{ E-mail addresses: \textsf{felix.delteso@uam.es}, \quad  \textsf{juanluis.vazquez@uam.es} } }
\vspace{1cm}

\date{}
\end{center}

\begin{abstract}
We formulate a numerical method to solve the porous medium type equation with fractional diffusion
\[\displaystyle\frac{\partial u}{\partial t}+\lap (u^m)=0\]
posed for $x\in \RR$, $t>0$, with $m\geq 1$, $\sigma \in (0,2)$, and nonnegative initial data $u(x,0)$. We prove existence and uniqueness of the solution of the numerical method and also the convergence to the theoretical solution of the equation with an order depending on $\sigma$. We also propose a two points approximation to a $\sigma$-derivative with order $O(h^{2-\sigma})$.
\end{abstract}


\section{Introduction}
In this paper we discuss a numerical method to solve the Cauchy problem
\begin{equation}\label{problem}
\left\{
\begin{array}{ll}
\displaystyle\frac{\partial u}{\partial t}+\lap ( |u|^{m-1}u)=0,& x\in \RR,\ t>0,\\[3mm]
u(x,0)=f(x), & x \in \RR,
\end{array}
\right.
\end{equation}
for exponents $m\geq1$, space dimension $N\geq1$, and fractional exponent $\sigma \in (0,2)$. After presenting the numerical method,  we prove existence and uniqueness of solution to the method. Moreover, we establish the convergence of the method to the theoretical solution of the problem. In the limit $\sigma \to 2$ we recover the standard Porous Medium Equation
\[
\displaystyle\frac{\partial u}{\partial t}+\Delta ( |u|^{m-1}u)=0,
\]
for which the numerical solution has been studied by many authors, either in itself of as part of the study of the class degenerate parabolic equations,  see e.g. \cite{awz, ciarlet, dibhoff, gravjam, hoffluc, jager, McCamy, NocVer, Rose, tm} for the earlier literature.

We recall that the fractional Laplacian operator $\lap$, $0<\sigma <2$, is probably the best known example in the class of nonlocal diffusion operators that are studied because of their interest both in theory and applications, cf. \cite{landkof, Stein}. Indeed, fractional diffusions have a long history in modeling problems in physics, finance, mathematical biology and hydrology.  The fractional Laplacian operator is usually defined via Fourier transform for any function $f$ in the Schwartz class as the operator such that
\[
\mathcal{F}(\lap u)(\xi)=|\xi|^\sigma\mathcal{F}(u)(\xi)\normalcolor
\]
or via Riesz potential, for a more general class of functions, as
\begin{equation}\label{sing.int}
\lap f(x)=C_{N,\sigma} \mbox{P.V.} \int_{\RR}\frac{f(x)-f(y)}{|x-y|^{N+\sigma}}dy\,,
\end{equation}
where $C_{N,\sigma}=2^{\sigma-1}\sigma \Gamma\left(\frac{N+\sigma}{2}\right)/\pi^{N/2} \Gamma\left(1-\frac{\sigma}{2}\right)$ is a normalization constant. For an equivalence of both formulations see for example \cite{val}.

The numerical analysis of  Problem \eqref{problem} was started in paper \cite{teso} where the case $\sigma=1$ was studied. The extension method that we use to implement the fractional Laplacian has a number of specific difficulties for $\sigma\ne 1$ that we address here. Previous works  dealing with the  numerical analysis of nonlocal equations of this type are due to Cifani, Jakobsen, and Karlsen in \cite{jakob}, \cite{jakob2}, \cite{jakob3}. In particular, they formulate some convergent numerical methods for entropy and viscosity solutions.

One of the main differences of our work is that we do not directly deal  with the integral formulation of the fractional  Laplacian;  instead of this, we pass through the  Caffarelli-Silvestre extension \cite{caff}, which replaces the calculation of the singular integral \eqref{sing.int} by the calculation of a convenient extension in one more space dimension. While in the case $\sigma=1$  the extension of $u^m(x,t)$ is just a function $w(x,y,t)$ which is harmonic in $(x,y)$ for every $t$ fixed, in the cases $\sigma\ne 1$ the extension is a so-called $\sigma$-harmonic function, i.\,e., the solution of an elliptic equation with a weight that is either degenerate or singular at $y=0$. The numerical analysis of the elliptic PDE $(-\Delta)^s u =f$ in a bounded domain with zero boundary data via the extension method has been recently studied by Nochetto and collaborators using  finite elements, \cite{NOS13}.

The paper is organized as follows. In {Section 2} we give a brief description of the problem we are concerned with. We present an equivalent way of expressing the problem avoiding the nonlocal operator formulation. For numerical reasons it is convenient to start by posing the problem in a bounded domain. In {Section 3} we propose a two-points approximation formula for the weighted derivative that appears in the extension formulation,  as well as a proof of the order of convergence; besides, a numerical experiment is given for this approximation. In {Section 4} we present the numerical scheme and prove convergence towards the theoretical solution. In {Sections 5 and 6} we study the optimal and minimal rates of convergence that we can reach with the proposed numerical scheme. Finally, in {Section 7} we show how the numerical solution posed in a bounded domain $\Omega$ converges to the theoretical solution posed in $\RR$ as $\Omega\to \RR$.



\section{Local formulation of the non-local problem}
\subsection{Problem in  $\RR$}
Our aim is to find numerical approximations for the solutions   of the Cauchy Problem for the porous medium equation with fractional diffusion, stated in \eqref{problem}.  We will take $m\geq 1$, $\sigma \in(0,2)$, and the initial function $f\in L^1\cap L^\infty(\RR)$ and nonnegative. The general theory for existence, uniqueness and regularity of the solution of problem (\ref{problem}) can be found in \cite{afracpor2}. In particular, it is shown that problem (\ref{problem}) is equivalent to the so-called extension formulation,
\begin{equation}\label{problemext}\displaystyle
\left\{ \begin{array}{ll}\displaystyle
L_\sigma w(x,y,t)=0,& x \in \RR, \ y>0,\ t>0, \\[3mm]
\displaystyle\frac{\partial w^{1/m}}{\partial t}=\frac{\partial w}{\partial y^{\sigma}},& x \in \RR,\ y=0,\ t>0, \\[3mm]
w(x,0,0)=f^m(x) , & x\in \RR,
\end{array} \right.
\end{equation}
where the extension is defined in terms of the elliptic operator:
$$
L_\sigma v:= \nabla \cdot (y^{1-\sigma} \nabla v)\,,
 $$
while the {\sl normalized $\sigma$-derivative operator}  $\displaystyle \frac{\partial}{\partial y^\sigma}$  is defined, for any $\sigma \in (0,2)$, as
\begin{equation}\label{extoperator}
\frac{\partial v}{\partial y^\sigma}(x,0):=\mu_\sigma\lim_{y\to0} y^{1-\sigma } \frac{\partial v}{\partial y}(x,y), \ \ \ \ \ .
\end{equation}
where $\mu_\sigma =2^{\sigma-1}\Gamma(\sigma/2)/\Gamma(1-\sigma/2)$. The equivalence between (\ref{problem}) and (\ref{problemext}) holds in the sense of trace and $L_\sigma$-harmonic extension operators, that is,
\[ u(x,t)=Tr(w^{1/m}(x,y,t)), \ \ \  w(x,y,t)=E_\sigma(u^m(x,t)).\]

\begin{nota}
The extensi\'on operator $L_\sigma$ defined in {\rm (\ref{extoperator})} can also be written as
\[
L_\sigma v(x,y)= y^{1-\sigma}\Delta v(x,y) + (1-\sigma)y^{-\sigma}\frac{\partial v}{\partial y}(x,y)\,,
 \]
where $\Delta$ is the $N+1$ dimensional Laplacian operator.
\end{nota}


\subsection{The problem in a bounded domain. Notations}
In order to construct a numerical solution to Problem (\ref{problemext}), we perform a monotone approximation of the solution in the whole space by the solutions of the problem posed in a bounded domain.

We consider positive numbers $X_1,\dots, X_N,Y,T$. We define the bounded domain $\Omega=(-X_1,X_1)\times . . . \times (-X_N,X_N)\times(0,Y)$, and set $\Gamma=\partial \Omega$. For convenience we also divide the boundary in two parts: $$
\Gamma_d=[-X_1,X_1]\times . . . \times [-X_N,X_N]\times \{0\}
$$
(the base), and $\Gamma_h=\partial \Omega \backslash \Gamma_d $ (the lateral boundary of the extended domain). With these notations, we formulate the corresponding problem in the bounded domain as
\begin{equation}\label{problemextbdd}\displaystyle
\left\{ \begin{array}{ll}\displaystyle
L_\sigma w(x,y,t)=0,& (x,y) \in \Omega ,\ t\in(0,T], \\[3mm]
\displaystyle\frac{\partial w^{1/m}}{\partial t}(x,0,t)=\frac{\partial w}{\partial y^{\sigma}}(x,0,t),& (x,y) \in \Gamma_d,\ t\in(0,T], \\[3mm]
w(x,0,0)=f^m(x) , & (x,y)\in \Gamma_d,\\[3mm]
w(x,y,t)=0, & (x,y)\in \Gamma_h\,.
\end{array} \right.
\end{equation}
Notice that we have imposed homogeneous boundary conditions on $\Gamma_h$.

In the sequel, we will consider the problem with $N=1$ in order to simplify de notation, but all the arguments are also valid for $N>1$ without much effort.

\begin{picture}(350,200)

\put(0, 15){\line(1, 0){350 }}

\put(175, 0){\line(0, 1){180 }}

\put(20,15){\dashbox{5}(310,140){}}

\put(20, 15.5){\line(1, 0){310 }}
\put(20, 14.5){\line(1, 0){310 }}

\put(20,155.5){\dashbox{5}(310,0){}}
\put(20,154.5){\dashbox{5}(310,0){}}

\put(19.5,15){\dashbox{5}(0,140){}}
\put(20.5,15){\dashbox{5}(0,140){}}

\put(330.5,15){\dashbox{5}(0,140){}}
\put(329.5,15){\dashbox{5}(0,140){}}

\put(325,0){\shortstack[l]{X}}

\put(10,0){\shortstack[l]{-X}}

\put(125, 80){\shortstack[l]{\huge$\Omega$}}

\put(180,160){\shortstack[l]{Y}}

\put(250, 180){\line(1, 0){40 }}
\put(250, 180.5){\line(1, 0){40 }}
\put(250, 179.5){\line(1, 0){40 }}
\put(295,178){$\Gamma_d$}

\put(250, 170){\dashbox{5}(40,0){}}
\put(250, 169.5){\dashbox{5}(40,0){}}
\put(250, 170.5){\dashbox{5}(40,0){}}
\put(295,166){$\Gamma_h$}

\end{picture}

\

 Let us announce at this point some other notations that will appear in the sequel:
 $\Lambda=\max_{i,k,j}|(\tau_j)_i^k|$ \ will be the local truncation error of the numerical method, while $a$ will be the discretization order for the numerical implementation $L^D_\sigma$ of the extended operator $L_\sigma$, cf. Section \ref{discreteform};  $E_j$ and $F_j$ are the two options for the total error of the numerical method at time $t_j$, which are used in the convergence proofs of Subsection \ref{errors}. We will use the Landau notation $O(.)$ for the order of a function relative to another one. Some constants appear: $\musig$ is defined after \eqref{extoperator}, and  $\nusig=\sigma\musig$


\section{Discretization of the $\sigma$-derivative}
Given the $\sigma$-harmonic extension problem
\begin{equation}\label{sigmaextension}
\left\{
\begin{array}{ll}
L_\sigma v(x,y)=\nabla \cdot (y^{1-\sigma}\nabla v)=0& x\in \RR,\ t>0,\\[3mm]
v(x,0)=g(x) & x \in \RR,
\end{array}
\right.
\end{equation}
the explicit solution is given by a convolution of the boundary condition with the kernel
\begin{equation}
P(x,y)=d_{N,\sigma}\frac{y^\sigma}{(|x|^2+y^2)^{\frac{N+\sigma}{2}}},
\end{equation}
that is,
\begin{equation}
v(x,y)=\int_{\RR}P(x-\xi,y)g(\xi)d\xi.
\end{equation}
In view of this, we introduce the discretized $\sigma$-derivative at $y=0$ as follows:
\begin{equation}\label{discret}
F(x,y):=\sigma \frac{v(x,y)-v(x,0)}{y^{\sigma}}=\sigma d_{N,\sigma}\int_{\RR}\frac{g(\xi)-g(x)}{(|x-\xi|^2+y^2)^{\frac{N+\sigma}{2}}}d\xi\,.
\end{equation}
With this definition, we have
\begin{equation}
\lim_{y\to0}F(x,y)=\lim_{y\to0} y^{1-\sigma}\frac{\partial v}{\partial y}(x,y).
\end{equation}
Notice that the operator $\frac{\partial}{\partial y^\sigma }$ satisfies
\[
\frac{\partial v}{\partial y^\sigma}(x,0):=\mu_{\sigma}\lim_{y\to0} y^{1-\sigma}\frac{\partial v}{\partial y}(x,y)=-\lap g(x).
\]
Summing up, it seems that $\mu_{\sigma}F(x,y)$ could be a good candidate to be used as the discretization of $\frac{\partial v}{\partial y^\sigma}(x,0)$.

We are interested in the order of the discretization. For that we have to compute the difference
\begin{eqnarray*}
\mu_{\sigma}  F(x,y)-\frac{\partial v}{\partial y^\sigma}(x,0)&=&\mu_\sigma \sigma d_{N,\sigma}\int_{\RR}\frac{g(\xi)-g(x)}{(|x-\xi|^2+y^2)^{\frac{N+\sigma}{2}}}d\xi+\lap g(x)\\[3mm]
&=&C_{N,\sigma}\int_{\RR}[g(\xi)-g(x)]\bigg(\frac{1}{(|x-\xi|^2+y^2)^{\frac{N+\sigma}{2}}}-\frac{1}{|x-\xi|^{N+\sigma}}\bigg)d\xi
\end{eqnarray*}

\begin{teor}\label{orderdiscretization}
Consider $F$ and $\displaystyle \frac{\partial v}{\partial y^{\sigma}}$ as above, and $g\in C^2(\RR)$. Then,
\[
|\mu_\sigma F(x,y)-\frac{\partial v}{\partial y^\sigma}(x,0)|\leq O(y^{2-\sigma}),
\]
that is,
\[|\musig \sigma \frac{v(x,y)-v(x,0)}{y^{\sigma}} - \frac{\partial v}{\partial y^\sigma}(x,0)|\leq O(y^{2-\sigma}).\]
\end{teor}

\begin{proof}
We want to have an estimate in terms of $y$ of the expression
\[
I=\frac{1}{C_{N,\sigma}}\left(\musig F(x,y)-\frac{\partial v}{\partial y^\sigma}(x,0)\right)=\int_{\RR}[g(\xi)-g(x)]\bigg(\frac{1}{(|x-\xi|^2+y^2)^{\frac{N+\sigma}{2}}}-\frac{1}{|x-\xi|^{N+\sigma}}\bigg)d\xi
\]
We split the integral into \ $I_R + I_C$, where $I_R$ is the above integral computed in $B_R(x)$ and $I_C$ the integral computed in $B_R^c :=\RR\backslash B_R(x)$. For simplicity, we are going to estimate all the integrals when $x=0$, but the calculation for a general $x$ is analogous.

Let us first estimate the integral outside the origin, $I_C$:
\begin{eqnarray*}
|I_C|&\leq& 2||g||_{\infty} \int_{B_R^c}\frac{1}{|\xi|^{N+\sigma}} - \frac{1}{(|\xi|^2+y^2)^{\frac{N+\sigma}{2}}}d\xi\\[2mm]
&=&\frac{C_g}{y^{\sigma}}\int_{B_{R/y}^c} \frac{(|z|^2+1)^{\frac{N+\sigma}{2}}-|z|^{N+\sigma}}{|z|^{N+\sigma}(|z|^2+1)^{\frac{N+\sigma}{2}}}dz\\[2mm]
&=&\frac{C_{g,N}}{y^{\sigma}}\int_{R/y}^\infty  \frac{(r^2+1)^{\frac{N+\sigma}{2}}-r^{N+\sigma}}{r^{\sigma+1}(r^2+1)^{\frac{N+\sigma}{2}}}dr=\frac{C_{g,N}}{y^{\sigma}}\int_{R/y}^\infty  \frac{f(1)-f(0)}{r^{\sigma+1}(r^2+1)^{\frac{N+\sigma}{2}}}dr.
\end{eqnarray*}
We first use the change of variables $\xi = yz$ and then a change to polar coordinates. We also define, for a fixed $r$ the function
\[f(x)=(r^2+x)^{\frac{N+\sigma}{2}},\]
with derivative
\[f'(x)=\frac{N+\sigma}{2}(r^2+x)^{\frac{N+\sigma-2}{2}}.
\]
Then $f(1)-f(0) =f'(\eta)$ for some $\eta \in[0,1]$. Now we need to consider two separate cases.

\medskip

\noindent $\bullet$ {If $\mathbf {N+\sigma \geq 2}$,} then,
 \[f'(x)\leq \frac{N+\sigma}{2}(r^2+1)^{\frac{N+\sigma-2}{2}} \ \ \forall x\in [0,1].\]

 In this way,
 \begin{eqnarray*}
 |I_C|&\leq& \frac{C_{g,N,\sigma}}{y^{\sigma}}\int_{R/y}^\infty  \frac{(r^2+1)^{\frac{N+\sigma-2}{2}}}{r^{\sigma+1}(r^2+1)^{\frac{N+\sigma}{2}}}dr=\frac{C_{g,N,\sigma}}{y^{\sigma}}\int_{R/y}^\infty  \frac{1}{r^{\sigma+1}(r^2+1)}dr\\[2mm]
 &\leq&\frac{C_{g,N,\sigma}}{y^{\sigma}}\int_{R/y}^\infty  \frac{1}{r^{\sigma+3}}dr=\frac{C_{g,N,\sigma}}{y^{\sigma}}\left(\frac{y}{R}\right)^{\sigma+2}=\frac{C_{g,N,\sigma}}{R^{\sigma+2}}y^2.
 \end{eqnarray*}

\medskip

\noindent $\bullet$ {If $\mathbf {N+\sigma < 2}$,} that is, $N=1$ and $\sigma <1$ then,
 \[f'(x)\leq \frac{1+\sigma}{2}r^{\sigma-1} \ \ \forall x\in [0,1].\]

 In this way,
  \begin{eqnarray*}
 |I_C|&\leq& \frac{C_{g,\sigma}}{y^{\sigma}}\int_{R/y}^\infty  \frac{r^{\sigma-1}}{r^{\sigma+1}(r^2+1)^{\frac{1+\sigma}{2}}}dr=\frac{C_{g,\sigma}}{y^{\sigma}}\int_{R/y}^\infty  \frac{1}{r^{2}(r^2+1)^{\frac{1+\sigma}{2}}}dr\\[2mm]
 &\leq&\frac{C_{g,\sigma}}{y^{\sigma}}\int_{R/y}^\infty  \frac{1}{r^{\sigma+3}}dr=\frac{C_{g,\sigma}}{R^{\sigma+2}}y^2.
 \end{eqnarray*}

This means that for every $N\geq1$ and $\sigma \in (0,2)$,
\begin{equation}
|I_C|\leq \frac{C_{g,N,\sigma}}{R^{\sigma+2}}y^2.
\end{equation}

And now, we estimate the integral inside de ball of radius $R$, $I_R$: since $ \frac{\nabla g(0)\cdot \xi}{(|\xi|^2+y^2)^\frac{N+\sigma}{2}}$ and $ \frac{\nabla g(0)\cdot \xi}{|\xi|^{N+\sigma}}$ are both odd functions with respect to $\xi$, they integrate zero in any ball centered in the origin. In this way,
\begin{eqnarray*}
|I_B|&=&\left| \int_{B_R}[g(\xi)-g(0)- \nabla g(0)\cdot \xi]\bigg(\frac{1}{(|\xi|^2+y^2)^{\frac{N+\sigma}{2}}}-\frac{1}{|\xi|^{N+\sigma}}\bigg)d\xi \right|\\[2mm]
&\leq&||D^2 g||_{\infty} \int_{B_R}|\xi|^2\bigg(\frac{1}{|\xi|^{N+\sigma}}-\frac{1}{(|\xi|^2+y^2)^{\frac{N+\sigma}{2}}}\bigg)d\xi\\[2mm]
&=&C_{g,N}\int^{R}_0  r^{1-\sigma}\frac{(r^2+y^2)^{\frac{N+\sigma}{2}}-r^{N+\sigma}}{(r^2+y^2)^{\frac{N+\sigma}{2}}}dr.
\end{eqnarray*}
With the same trick with the function $f$ defined above, we need to split the calculation in two cases again.

\medskip

\noindent $\bullet$ {If $\mathbf {N+\sigma \geq 2}$,} then,
\begin{eqnarray*}
|I_B|&\leq& C_{g,N,\sigma}\int^{R}_0  r^{1-\sigma}\frac{(r^2+y^2)^{\frac{N+\sigma-2}{2}}y^2}{(r^2+y^2)^{\frac{N+\sigma}{2}}}dr=C_{g,N,\sigma}y^2\int^{R}_0  r^{1-\sigma}\frac{1}{r^2+y^2}dr\\[2mm]
&=&C_{g,N,\sigma}y^2\int^{R/y}_0  y^{1-\sigma}s^{1-\sigma}\frac{y}{y^2(s^2+1)}ds=C_{g,N,\sigma}y^{2-\sigma}\int^{R/y}_0 \frac{s^{1-\sigma}}{(s^2+1)}ds.\\[2mm]
&\leq&C_{g,N,\sigma}y^{2-\sigma}\int^{\infty}_0 \frac{s^{1-\sigma}}{(s^2+1)}ds=d_{g,N,\sigma}y^{2-\sigma},
\end{eqnarray*}
where we have used the change of variables $r=sy$ and the integrability of the function $\frac{s^{1-\sigma}}{(s^2+1)}$ for every $\sigma \in (0,2)$.

\medskip

\noindent $\bullet$ {If $\mathbf {N+\sigma < 2}$,} that is, $N=1$ and $\sigma<1$ then,
\begin{eqnarray*}
|I_B|&\leq&C_{g}\int^{R}_0  r^{1-\sigma}\frac{(r^2+y^2)^{\frac{1+\sigma}{2}}-r^{1+\sigma}}{(r^2+y^2)^{\frac{1+\sigma}{2}}}dr= C_{g,\sigma}\int^{R}_0  r^{1-\sigma}\frac{r^{1+\sigma-2}y^2}{(r^2+y^2)^{\frac{1+\sigma}{2}}}dr\\[2mm]
&=&C_{g,\sigma}y^2\int^{R}_0  \frac{1}{(r^2+y^2)^{\frac{1+\sigma}{2}}}dr= C_{g,\sigma}y^2\int^{R/y}_0  \frac{y}{y^{1+\sigma}(s^2+1)^{\frac{1+\sigma}{2}}}ds\\[2mm]
&=&C_{g,\sigma}y^{2-\sigma}\int^{R/y}_0  \frac{1}{(s^2+1)^{\frac{1+\sigma}{2}}}ds=C_{g,\sigma}y^{2-\sigma}\int^{\infty}_0  \frac{1}{(s^2+1)^{\frac{1+\sigma}{2}}}ds\\[2mm]
&=& d_{g,\sigma} y^{2-\sigma}
\end{eqnarray*}
where again we have used the change  $r=sy$ and the integrability of the function $ \frac{1}{(s^2+1)^{\frac{1+\sigma}{2}}}$ for every $\sigma >0$.
\end{proof}

\begin{nota}
Notice that the case $\sigma\not = 1$ generalizes the usual result of the order of discretization of the forward Euler discretization for a first derivative.
\end{nota}
\noindent {\bf Numerical experiments. } Let us show that the results given in Theorem \ref{orderdiscretization} are in fact optimal, i.\,e.,  that we could not expect to have an order better than $2-\sigma$ with the discretization given by $F(x,y)$ in (\ref{discret}).
We choose as test function the following function that does not depend on $x$,
\begin{equation}
f(x,y)=e^{y^2},
\end{equation}

 It is easy to see that for every $\sigma \in (0,2)$ we have that
\[
\sigma\frac{\partial f}{\partial y^\sigma}(x,0)= \lim_{y\to 0} y^{1-\sigma}\frac{\partial f}{\partial y}(x,y)=\lim_{y\to 0}2y^{2-\sigma}e^{y^2}=0.\]
The error $E$ for the discretization $F$ will be
\[E(y)=|F(x,y)-0|=|F(x,y)|.\]

\begin{table}[h!]
\caption{Numerical results for $\sigma=1/2, 1 \mbox{ and }3/2$.}
        		\label{numres}
$$
 \begin{tabular}{|c||c|c|c|c|}
  \hline
 $\sigma$ & $y$&$E(y)$ & $\alpha$&$\sigma_e$\\
 \hline
 \hline
  1	&  1/2	&0.5681     & -&     -          \\
   	&   1/4 	& 0.2580           &1.1388   &      0.8612              \\
   	&  1/8	&0.1260    &    1.0340 &    0.9660        \\
     	&  1/16	& 0.0626     &  1.0085&       0.9915       \\
  \hline
  \hline
  \hline
  1/2	&  1/2	&0.2008       &     -        &      -         \\
   	&   1/4 	&0.0645     &1.6388 &      0.3612              \\
   	&  1/8	& 0.0223    &1.5340 &      0.4660       \\
     	&  1/16	 & 0.0078     &1.5085 &       0.4915       \\
  \hline
  \hline
  \hline
   3/2	&  1/2	&1.2050                &- &       -       \\
   	&   1/4 	&0.7739    &0.6388       &     1.3612              \\
   	&  1/8	&0.5345    &0.5340&       1.4660       \\
     	&  1/16	&0.3757    & 0.5085 &      1.4915     \\
\hline
 \end{tabular}
$$
\end{table}
The result of Theorem \ref{orderdiscretization} says that we have an error lower than $O(y^{2-\sigma})$, that is
\[E(y)= K y^\alpha,\]
for some $K>0$ and $\alpha=2-\sigma$.  In the next table we show the experimental results for different $\sigma$'s and $y$'s. The variable $\sigma_e$ denotes the $2-\alpha$, that is, the $\sigma$ given by the experiments.

The way of computing $\alpha$ is the usual: Given two errors $E(y_1)$ and $E(y_2)$,
\[
E(y_1)= K y_1^\alpha,\quad  E(y_2)= K y_2^\alpha,
\]
then,
\[
\alpha=\frac{\log\left(E(y_1)/E(y_2)\right)}{\log(y_1/y_2)}.
\]
The results are shown in Table \ref{numres}.

\section{Discrete formulation}\label{discreteform}

In order to solve problem (\ref{problemextbdd}) for $t\in[0,T]$, we first perform a time and space discretization.  For the time discretization we choose  $J$ uniformly spaced steps, and then $\Delta t=T/J$ and
\[
0\leq j \Delta t\leq T, \ \ \ j=0, . . .   ,J, \qquad  t_j=j\Delta t.
\]
We also need to discretize the space domain  $\overline{\Omega}=[-X,X]\times[0,Y]$. Let $I,K$ be the number  of steps on each space direction,
\[
0 \leq i \Delta x  \leq 2X , \ \ \ i=0, . . .   ,I \ \ \mbox{where } \Delta x=2X/I \mbox{ and } x_i= i \Delta x-X,
\]
\[
0 \leq k \Delta y \leq Y , \ \ \ k=0, . . .   ,K \ \ \mbox{where } \Delta y=Y/K \mbox{ and } y_k= k \Delta y.
\]

\begin{picture}(350,190)

\put(15, 15){\line(1, 0){320 }}

\put(175, 0){\line(0, 1){180 }}

\put(20,15){\dashbox{5}(310,140){}}

\put(20,15.5){\dashbox{5}(310,0){}}
\put(20,14.5){\dashbox{5}(310,0){}}

\put(20,155.5){\dashbox{5}(310,0){}}
\put(20,154.5){\dashbox{5}(310,0){}}

\put(19.5,15){\dashbox{5}(0,140){}}
\put(20.5,15){\dashbox{5}(0,140){}}

\put(330.5,15){\dashbox{5}(0,140){}}
\put(329.5,15){\dashbox{5}(0,140){}}

\put(125, 80){\shortstack[l]{\huge$\Omega$}}


\put(20,15){\grid(310,140)(15.5,15.5)}
\put(14,5){\shortstack[l]{$x_0$}}
\put(30,5){\shortstack[l]{$x_1$}}
\put(46,5){\shortstack[l]{$x_2$}}
\put(62,5){\shortstack[l]{$. . . $}}
\put(77,5){\shortstack[l]{$x_i$}}
\put(91,5){\shortstack[l]{$. . . $}}

\put(325,5){\shortstack[l]{$x_I$}}

\put(5, 15){\shortstack[l]{$y_0$}}
\put(5, 29){\shortstack[l]{$y_1$}}
\put(5, 43){\shortstack[l]{$y_2$}}
\put(8, 55){\shortstack[l]{$.$}}
\put(8, 59){\shortstack[l]{$.$}}
\put(8, 63){\shortstack[l]{$.$}}
\put(5, 70){\shortstack[l]{$y_k$}}
\put(8, 84){\shortstack[l]{$.$}}
\put(8, 88){\shortstack[l]{$.$}}
\put(8, 92){\shortstack[l]{$.$}}

\put(5, 155){\shortstack[l]{$y_K$}}
\end{picture}
\\\\
We use the notation
\begin{equation}\label{notation}
w(x_i,y_k,t_j)=(w_j)_i^k
\end{equation}
for the values of the  theoretical solution   $w$ to Problem (\ref{problemextbdd}) at the points of the mesh, and
\begin{equation}\label{notation}
w(x_i,y_k,t_j)\approx(W_j)_i^k
\end{equation}
for the solution of the numerical method.

\subsection{Numerical Method}\label{num.meth}

 Let us assume that $\Delta y=\Delta x$. We consider a discretization $L^D_\sigma$ of the operator $L_\sigma$ with a specified  order $a$ to be chosen later. This means that  the following estimate holds
\begin{equation}\label{orderdisc}
\max_{i,k}|L_\sigma^D v_i^k-L_\sigma v(x_i,y_k)|\leq O(\Delta x^{a})
\end{equation}
 for every  regular enough function $v$,  where we use the shortened form $\Delta x^a$ to denote $(\Delta x)^a$.
For each time step $j=1,. . . ,J,$ we have to solve the following linear system of equations
\begin{equation}\label{numericmethod}\displaystyle
\left\{ \begin{array}{ll}\displaystyle
L_\sigma^D \left[(W_j)_i^k\right] =0, &  0<i<I,
 0<k<K,\\[3mm]\displaystyle
(W_j)_i^0 =\bigg[\nusig\frac{\Delta t}{\Delta x^\sigma} \big({(W_{j-1})}_i^1-{(W_{j-1})}_i^0\big)  +[(W_{j-1})_i^0]^{1/m}\bigg]^m, & \mbox{if } 0<i<I,\\[3mm]\displaystyle
(W_j)_i^k=0,& \mbox{on the $\Gamma_h$ nodes. }
\end{array} \right.
\end{equation}
 Recall that $\nusig=\sigma\musig$ and note that the second equation is explicit in the sense that all the terms in the right-hand side of the equation are known from the previous step.  To start the numerical method we will use the solution of
\begin{equation}\nonumber\displaystyle
\left\{ \begin{array}{ll}\displaystyle
L_\sigma^D \left[(W_0)_i^k\right] =0, &  0<i<I,
 0<k<K,\\[3mm]\displaystyle
(W_0)_i^0=f^m(x_i), & \mbox{if } 0<i<I,\\[3mm]\displaystyle
(W_0)_i^k=0,& \mbox{on the $\Gamma_h$ nodes. }\\
\end{array} \right.
\end{equation}

\subsection{Local truncation error}
We define the local truncation error $(\tau_j)_i^k$ as the error that comes from plugging the solution $w$ to Problem (\ref{problemextbdd}) into the numerical method (\ref{numericmethod}).
Let us also write
\begin{equation}
\Lambda=\max_{i,k,j}|(\tau_j)_i^k|.
\end{equation}

\begin{teor}\label{localtruncerror}
Let $w$ be the solution to Problem {\rm (\ref{problemextbdd})}.
Then,
\begin{equation}
\Lambda=O\,(\Delta t(\Delta x^{2-\sigma}+\Delta t)+\Delta x^a).
\end{equation}
where $a$ is the order of discretization of $L_\sigma^D$.
\end{teor}

\begin{proof}
Of course, the local truncation error in the boundary nodes situated on the part $\Gamma_h$ of the boundary  is zero since we have imposed that the solution is zero in $\Gamma_h$ and equal to $f^m(x)$ in $\Gamma_d$ as in Problem (\ref{problemextbdd}).

If $0<i<I$ and $0<k<K$ (the interior nodes), then
\begin{eqnarray*}
(\tau_{j-1})_i^k : =L^D_\sigma \left[(w_j)_i^k\right]=
 L_\sigma w(x_i,y_k,t_j)+O(\Delta x^a)=O(\Delta x^a).
\end{eqnarray*}

If $0< i< I$ and $k=0$ (i.\,e., at the boundary  nodes $\Gamma_h$), the local truncation error is calculated as
\begin{eqnarray*}
(\tau_{j-1})_i^0&: =&\nusig\frac{\Delta t}{\Delta x^\sigma} \big[{(w_{j-1})}_i^1-{(w_{j-1})}_i^0\big]  +[(w_{j-1})_i^0]^{1/m}-[(w_{j})_i^0]^{1/m}\\
&=&\Delta t\big[ \frac{\partial w}{\partial y^\sigma}(x_i,0,t_{j-1})+O(\Delta x^{2-\sigma})\big]- \Delta t \big[\frac{\partial w^{1/m}}{\partial t}(x_i,0,t_{j-1})+O(\Delta t) \big]\\
&=&O(\Delta t \Delta x^{2-\sigma})+ O(\Delta t^2)=O(\Delta t(\Delta t+\Delta x^{2-\sigma})).
\end{eqnarray*}
The preceding calculation is done on the assumption that the theoretical solution is smooth, more precisely we use that $u$ is $C^2$ w.r.t. $t$ and $u^m$ is $C^2$ w.r.t. $x$. \end{proof}

\subsection{Existence and uniqueness of the numerical solution}

The quantity
\[
b_{max}= \max_{x}\{f^m(x)\}=\|f\|^m_\infty
\]
 appears below in the application of the maximum principle.  In the sequel we will often denote the power function $u^m$ by $\varphi(u)$. When $m\geq1$ then $\varphi'(u)=mu^{m-1}$ is a locally bounded function for $u\geq0$.

\begin{teor}\label{maxprin}[Discrete maximum principle]
Let $(W_j)_i^k$ be the solution to  Problem {\rm (\ref{numericmethod})} with $m\geq1$.  Assume that \
\begin{equation}\label{constant}
\Delta t\leq C(m,f) \Delta x^\sigma, \quad \mbox{ where} \quad C(m,f)=[m (b_{max})^{(m-1)}\nusig]^{-1}\,.
\end{equation}
Then for every $i,k,j$ we have
\begin{equation}
0\leq (W_j)_i^k \leq b_{max}.
\end{equation}
\end{teor}
\begin{nota}
It is interesting to remark that $C(1,f) =1$, which means that we recover the expected restriction $\Delta t\leq \Delta x^\sigma$ for the linear case.
\end{nota}

We start with a lemma.
\begin{lemma}\label{maxprinoper}
Let $(W_j)_i^k$ be the solution to  Problem {\rm (\ref{numericmethod})} with $m\geq1$. Then, for a fixed time $t_j=j\Delta t$, we have
\[\max_{i,k}\{(W_j)_i^k\}=\max_{i}\{(W_j)_i^0\}.\]
that is, the maximum is always attained on the boundary $\Gamma_d$.
\end{lemma}
\begin{proof}
We recall that we are using the following extension operator,
\[
L_\sigma w(x,y)=\nabla\cdot(y^{1-\sigma}\nabla w )=y^{1-\sigma}\Delta w + (1-\sigma)y^{-\sigma}\frac{\partial w}{\partial y}.
\]
In two dimensions, the lowest order discretization we can consider is,
\begin{equation}\label{discext}
L_\sigma^DW_i^k=y_k^{1-\sigma}\frac{W_{i+1}^k+W_{i-1}^k+W^{k+1}_i+W^{k-1}_i-4W_i^k}{\Delta x^2}+\frac{(1-\sigma)}{y_k^\sigma}\frac{W_i^{k+1}-W_i^k}{\Delta x}.
\end{equation}
Assume that there exists an interior node $(i,k)$ such that
\[W_i^k=\max_{a,b}\{W_a^b\},\]
that is, the maximum is attained there. Note that, since we are talking about an interior node, $k\geq1$ and we can write $y_k=k\Delta x$.  Then equation $L_\sigma^D W_i^k=0$ obtained from (\ref{discext}) becomes,
\[
k^{1-\sigma}\frac{W_{i+1}^k+W_{i-1}^k+W^{k+1}_i+W^{k-1}_i-4W_i^k}{\Delta x^{1+\sigma}}+\frac{(1-\sigma)}{k^\sigma}\frac{W_i^{k+1}-W_i^k}{\Delta x^{1+\sigma}}=0.
\]
and so,
\begin{eqnarray*}
0&\leq& k^{1-\sigma}\frac{3W_{i}^k+W^{k+1}_i-4W_i^k}{\Delta x^{1+\sigma}}+\frac{(1-\sigma)}{k^\sigma}\frac{W_i^{k+1}-W_i^k}{\Delta x^{1+\sigma}}\\
&=&\frac{k}{k^{\sigma}}\frac{W^{k+1}_i-W_i^k}{\Delta x^{1+\sigma}}+\frac{(1-\sigma)}{k^\sigma}\frac{W_i^{k+1}-W_i^k}{\Delta x^{1+\sigma}}=[k+1-\sigma]\frac{W_i^{k+1}-W_i^k}{k^\sigma\Delta x^{1+\sigma}}.
\end{eqnarray*}
We recall that $k+1-\sigma >0$, so we conclude that $W_i^{k+1}\geq W_i^k$. Since $W_i^k$ is the maximum, we get that  $W_i^{k+1}= W_i^k$. At this point we proceed by induction on $k$ to get
\[W_i^k=W_i^{k+1}=W_i^{k+2}=\cdots=W_i^K.\]
But by hypothesis, $W_i^K=0$, so we get a contradiction.
\end{proof}

\begin{proof}[Proof of Theorem \ref{maxprin} ]
Using the result of the previous lemma, we only need to prove the maximum principle at the boundary $\Gamma_d$ nodes.
We will do the proof by induction on each time step. It is trivial that
\[0\leq (W_0)_i^k\leq b_{max}.\]
Now we assume that
\[0\leq (W_{j-1})_i^k\leq b_{max}.\]
Then,
\[[(W_j)_i^0]^{1/m}=\nusig\frac{\Delta t}{\Delta x^\sigma} \big({(W_{j-1})}_i^1-{(W_{j-1})}_i^0\big)  +[(W_{j-1})_i^0]^{1/m},\]
If we change variables to $\displaystyle(U_j)_i^k= [(W_j)_i^k]^{1/m} $ and use Mean Value Theorem , we obtain, for some $\xi \in [(U_{j-1})_i^1,(U_{j-1})_i^0]$ that
\begin{equation}\label{recu2}
(U_j)_i^0=\varphi'(\xi)\nusig\frac{\Delta t}{\Delta x^\sigma} (U_{j-1})_i^1 +\bigg[1-\varphi'(\xi)\nusig\frac{\Delta t}{\Delta x^\sigma}\bigg](U_{j-1})_i^0.
\end{equation}
At this point , thanks to our induction hypothesis and the value of the constant (\ref{constant}), it follows that $\displaystyle \varphi'(\xi)\nusig\frac{\Delta t}{\Delta x^\sigma} \leq 1$ and therefore
\begin{equation}
|(U_j)_i^0|\leq\varphi'(\xi)\nusig\frac{\Delta t}{\Delta x^\sigma} (b_{max})^{1/m} +\bigg[1-\varphi'(\xi)\nusig\frac{\Delta t}{\Delta x^\sigma}\bigg] (b_{max})^{1/m}
=(b_{max})^{1/m}.
\end{equation}
The same argument holds for $ (U_j)_i^0\geq 0$.
\end{proof}

\begin{coro}\label{existence}
If $\Delta t\leq C(m,f) \Delta x^\sigma$, then Problem {\rm (\ref{numericmethod})} has a unique solution.
\end{coro}
 We are a bit sketchy with these rather standard proofs of Theorem \ref{maxprin} and Corollary  \ref{existence}, the reader can consult a more detailed exposition in in the paper \cite{teso} where
 the case $\sigma=1$ is covered. Basically, two solutions $(V_j)_i^k$ and $(W_j)_i^k$ with the same initial condition are considered. Then $(Z_j)_i^k= (V_j)_i^k-(W_j)_i^k$ is also a solution with zero initial condition. By the maximum principle we have $(Z_j)_i^k=0$, hence $ (V_j)_i^k=(W_j)_i^k$. Since, for a linear system of equations with the same number of unknowns and equations, existence is equivalent to uniqueness, the required result is proved.


\subsection{Error of the numerical method and convergence to the solution}\label{errors}
Since we are originally interested in the values of the numerical solution at the boundary, $(U_j)_i^k=[(W_j)_i^k]^{1/m}$, we have two options in order to define the {\sl error of the numerical method}. The first option is through the $w$ variables:
\begin{equation}\label{error2}
(f_j)_i^k=w(x_i,y_k,t_j)-(W_j)_i^k, \ \ \ F_j=\max_{i,k}|(f_j)_i^k|,
\end{equation}
and the second one through the $u$'s:
\begin{equation}\label{error1}
(e_j)_i^k=u(x_i,y_k,t_j)-(U_j)_i^k, \ \ \ E_j=\max_{i,k}|(e_j)_i^k|.
\end{equation}
Anyway, if we are able to control (\ref{error1}) we have also a control of (\ref{error2}) because $\varphi'(x)$ is locally bounded and
\begin{eqnarray*}
(f_j)_i^k&=&(w_j)_i^k-(W_j)_i^k=[(u_j)_i^j]^m-[(U_j)_i^k]^m\\
&=&\big[(u_j)_i^k-(U_j)_i^k\big]\varphi'(\xi)\\
&=& (e_j)_i^k\varphi'(\xi),
\end{eqnarray*}
for some $\xi$ in the interval $[(u_j)_i^k,(U_j)_i^k]$. This implies that $|(f_j)_i^k|\leq C(m,f)|(e_j)_i^k|$, and therefore $F_j\leq C(m,f) E_j$.

\begin{teor}\label{convergence}
Let $w$ be the classical solution to Problem {\rm (\ref{problemextbdd})} and $(W_j)_i^k$ be the solution to System {\rm (\ref{numericmethod})} with $m\geq 1$, $\sigma\in(0,2)$ and a discretization $L_\sigma^D$ with order $a$ (as in Subsection {\rm 4.1}). Assume that
\begin{equation}
\Delta t\leq C(m,f)\Delta x^\sigma.
\end{equation}
Then,
\begin{equation}
F_j=O(\Delta t+ \Delta x^{2-\sigma}+\frac{\Delta x^a}{\Delta t}),
\end{equation}
for $j=1,. . .,J$. Therefore, if $\Delta x^a/\Delta t\to 0$, then  the numerical solution $(W_j)_i^k$ converges to the theoretical solution $w$ as $\Delta x, \Delta t \to 0$.
\end{teor}

To prove this theorem we will first study how the errors are propagated from the boundary to the interior on each time step. We consider the solution of the  numerical scheme $(W_j)_i^k$ and the theoretical one $(w_j)_i^k$ on time $t=j\Delta t$, and we will call them for simplicity $W_i^k$ and $w_i^k$. Moreover, in the following lemma we will use the notation $h=\Delta x$.
\begin{lemma}\label{errorpropagation}
Let $L_h$ be a discretization of order $a$ of the operator
\[Lw=-\nabla \cdot (y^{1-\sigma}\nabla w)\]
which satisfies the maximum principle as in Lemma {\rm \ref{maxprinoper}}. Then, for some constant $K>0$,
\[\max_{i,k}|w_i^k- W_i^k| \leq \max_{i}|w_i^0- W_i^0| + Kh^a.\]
This means that the error is propagated to the interior of the domain as $O(h^a)$.
\end{lemma}

\begin{proof} Note that this a calculation in the space variables for fixed time. Since the problem is linear we may split it in two contributions, and consider only the nontrivial part of the influence of the right-hand side of the equation.

\noindent (i) Consider a function $f\in C^\infty (\Omega)$ such that $f\geq0$.  Assume that $f$ is a strict supersolution of $L$ in the sense that there exists a constant $D>0$ such that
\begin{equation}
L[f]\geq D>0.
\end{equation}
Since $f\in C^\infty(\Omega)$, also $w\in C^\infty(\Omega)$ and $L_h$ is a discretization of order $a$, we have
\[
\left|L[f]-L_h[f] \right|\leq C_1h^a, \qquad \left|L_h[w] \right|\leq C_2h^a,
\]
for some constants $C_1=C_1(f)>0$ and $C_2=C_2(w)>0$.  In particular,
\begin{equation}\label{orderdiscf}
L_h[f] \geq L[f]-C_1h^a.
\end{equation}
Consider also the function $\displaystyle g=\frac{3C}{D}h^a f$ with  $C=\max\{C_1,C_2\}$.. Then $g$ is also a super solution such that
\[L[g]=\frac{3C}{D}h^a L[f] \geq \frac{3C}{D}h^a D=3C h^a>0. \]
By (\ref{orderdiscf}) we have,

\[L_h[g]\geq L[g]-Ch^a\geq 3Ch^a-Ch^a=2Ch^a.\]

Consider now $e_i^k=w_i^k -W_i^k$ as usual. In this way,
\begin{eqnarray*}
L_h[e_i^k-g]&=&L_h[e_i^k]-L_h[g]=L_h[w_i^k]- L_h[W_i^k]-L_h[g]\\
&\leq& C h^a - 2Ch^a=-Ch^a<0,
\end{eqnarray*}
where we have use that $L_h[W_i^k]=0$ by hypothesis and $L_h[w_i^k]=O(h^a)$ because $L_h$ is a discretization of order $a$. Then we can conclude by the maximum principle of the discretized operator $L_h$ that
\[e_i^k\leq g\leq K h^a\]
where $K=\frac{3C}{D}||f||_\infty$. A similar argument is used to obtain $e_i^k \geq -Kh^a$ and so on $|e_i^k| \leq Kh^a$ for all the interior nodes.

\medskip

\noindent (ii) We still have to find an explicit positive strict supersolution $f\in C^\infty(\Omega)$ of the operator $L=-\nabla \cdot (y^{1-\sigma}\nabla w)$.

 In case $\sigma \in (0,1)$ we consider the function $f(x,y)=Y-y$, where $Y$ is the vertical length the domain $\Omega=[-X,X]\times[0,Y]$. It is clear that $f\geq 0$ and $f\in C^{\infty}(\Omega)$ so it is under the regularity assumptions of our lemma. Then
\[
\nabla f=[0,-1], \quad y^{1-\sigma}\nabla f=[0,-y^{1-\sigma}], \quad  -\nabla\cdot\left(y^{1-\sigma}\nabla f\right)=(1-\sigma)\frac{1}{y^{\sigma}}.
\]
Since $y\in (0,Y)$, we have that $-\nabla\cdot\left(y^{1-\sigma}\nabla f\right) \geq (1-\sigma)\frac{1}{Y^{\sigma}}>0$.

For $\sigma \in [1,2)$ take $f(x,y)=Y^2-y^2$.  Then
\[\nabla f=[0,-2y], \quad y^{1-\sigma}\nabla f=[0,-2y^{2-\sigma}], \quad  -\nabla\cdot\left(y^{1-\sigma}\nabla f\right)=2(2-\sigma)\frac{1}{y^{\sigma-1}}.\]
and so $-\nabla\cdot\left(y^{1-\sigma}\nabla f\right) \geq2(2-\sigma)Y^{1-\sigma}>0$.
\end{proof}
\normalcolor

\begin{proof}[Proof of Theorem \ref{convergence}]
As in the local truncation error, the choice of the boundary conditions on the lateral boundary $\Gamma_h$ for our numerical method gives us zero error  there.

Lets us denote $(E_B)_j$ and $(E_I)_j$ the maximum errors in the boundary nodes and in the interior nodes at time $t_j=j\Delta t$, that is
\[
(E_B)_j=\max_{0\leq i\leq I}|(e_j)_i^0|.
\]
 As a consequence of Lemma \ref{errorpropagation} we have
\[
E_j=(E_B)_j+ O(\Delta x^a).
\]
If $0\leq i\leq I$, we have the following equations in terms of $u$ and $(U_j)_i^k$,
\[
(\tau_{j-1})_i^0=\nusig\frac{\Delta t}{\Delta x^\sigma} \big[[{(u_{j-1})}_i^1]^{m}-[{(u_{j-1})}_i^0]^{m}\big]  +(u_{j-1})_i^0-(u_{j})_i^0,
\]
\[(U_j)_i^0=\nusig \frac{\Delta t}{\Delta x^\sigma} \big[[{(U_{j-1})}_i^1]^{m}-[{(U_{j-1})}_i^0]^{m}\big]  +(U_{j-1})_i^0.
\]
Subtracting them, and using the Mean Value Theorem we get, for some $\xi_0\in [{(u_{j-1})}_i^0,{(U_{j-1})}_i^0\big]$ and $\xi_1\in [{(u_{j-1})}_i^1,{(U_{j-1})}_i^1\big]$, that
\[(e_j)_i^0=\nusig\frac{\Delta t}{\Delta x^\sigma}\varphi'(\xi_1){(e_{j-1})}_i^1  +\bigg[1-\nusig\frac{\Delta t}{\Delta x^\sigma}\varphi'(\xi_0)\bigg](e_{j-1})_i^0-(\tau_{j-1})_i^0.
\]
By assumption, all the coefficients that come with $(e_{j-1})_i^k$ are positive, therefore
\begin{eqnarray}\label{eqerror2}
|(e_j)_i^0|\leq\nusig\frac{\Delta t}{\Delta x^\sigma}\varphi'(\xi_1)E_{j-1}  +\bigg[1-\nusig\frac{\Delta t}{\Delta x^\sigma}\varphi'(\xi_0)\bigg]E_{j-1}+ \Lambda.
\end{eqnarray}
We now need  to control  the difference between $\varphi'(\xi_1)$ and $\varphi'(\xi_0)$. Assuming enough regularity of the solution $u$, there exists a constant $K\geq0$ such that
\[|(u_j)_i^1-(u_j)_i^0|\leq K\Delta x^\sigma\ \  \mbox{ and }\ \ |(U_j)_i^1-(U_j)_i^0|\leq K\Delta x^\sigma.\]
end then \ $|\varphi'(\xi_1)- \varphi'(\xi_0)|\leq R\Delta x^\sigma$, where $R\geq0$ is a constant depending only on $m$, $K$, $\sigma$  and $b_{max}$. The proof is only based in the idea of that  the function
\[
g(x,y)=\frac{x^m-y^m}{x-y}.
\]
is $C^1((0,\infty)\times (0\times \infty))$, and $\varphi'(\xi_0)=g((u_{j})_i^0,(U_{j})_i^0)$ and $\varphi'(\xi_1)=g((u_{j})_i^1,(U_{j})_i^1)$.
 Then from (\ref{eqerror2}) we obtain
\begin{eqnarray*}
|(e_j)_i^0| &\leq&\nusig\frac{\Delta t}{\Delta x^\sigma}\bigg[\varphi'(\xi_0)+ R \Delta x^\sigma \bigg]E_{j-1}  +\bigg[1-\nusig\frac{\Delta t}{\Delta x^\sigma}\varphi'(\xi_0)\bigg]E_{j-1}+\Lambda\\
&\leq& (1+D\Delta t)E_{j-1}+\Lambda.
\end{eqnarray*}

Remember also that we have
$$\displaystyle E_j\leq \max_{0<i<I} |(e_j)_i^0|+O(\Delta x^a)$$
 and  $\Lambda =O(\Delta x^a+\Delta t (\Delta t+\Delta x^{2-\sigma}))$.  Then (recalling that $J=T/\Delta t$), we have the following recurrence equation for the error in terms in $J$ ,
\begin{equation}\label{errorrec2}
E_j\leq (1+C\frac{1}{J})E_{j-1}+\Lambda,
\end{equation}
for some constant $C>0$. Of course, it is enough to bound $E_{J}$ to ensure the convergence of the method. In this way,
\begin{equation}
E_J\leq(1+C\frac{1}{J})\bigg[E_{J-1}+\Lambda\bigg]
\leq . . . \leq (1+C\frac{1}{J})^J\bigg[E_0+L J \Lambda\bigg].
\end{equation}

But $(1+C\frac{1}{J})^J\leq e^{C}$, $J\Delta t=T$ and $E_0\leq D \Delta x^{a} $ for some $D>0$, so
\[E_J\leq \frac{C}{\Delta t}\big(\Delta t(\Delta t +\Delta x^{2-\sigma})+\Delta x^a\big),\]
that is
\[E_J=O(\Delta t+\Delta x^{2-\sigma}+\frac{\Delta x^a}{\Delta t}).\]
\end{proof}

\subsection{ Practical application}

At this point, the best we can expect is to have an error of the form
\begin{equation}\label{errores}
E_j=O(\Delta t +\Delta x^{2-\sigma}).
\end{equation}
Let us look for a discretization of the extension operator which will not include any extra error to the scheme. Relation (\ref{errores}) gives two possible conditions to impose on $a$:
\begin{enumerate}
\item $\displaystyle\frac{\Delta x^a}{\Delta t} \leq \Delta t$.
\item $\displaystyle\frac{\Delta x^a}{\Delta t} \leq \Delta x^{2-\sigma}$.
\end{enumerate}
Note that both inequalities are only required to be true up to some constant. We need also these two conditions to be compatible with $\Delta t \leq C(m,f) \Delta x^{\sigma} $.
\begin{nota}
The function $g(\sigma)=\Delta x^\sigma$ is monotone decreasing for small fixed $\Delta x$, that is, is $\alpha\geq \beta>0$ then $\Delta x^\alpha\leq \Delta x^{\beta}$.
\end{nota}

\underline{Condition 1:} $\Delta x^{\frac{a}{2}}\leq \Delta t$.

We need $a$ for which there exists a $\Delta t$ such that
\begin{equation}
\Delta x^{\frac{a}{2}}\leq \Delta t\leq \Delta x^\sigma,
\end{equation}
that is, $a/2\geq \sigma$. So if
\fbox{$a\geq 2\sigma$}
this condition is compatible.

\underline{Condition 2:} $\Delta x^{a+\sigma-2}\leq \Delta t$.

We need $a$ for which there exists a $\Delta t$ such that
\begin{equation}
\Delta x^{a+\sigma-2}\leq \Delta t\leq \Delta x^\sigma,
\end{equation}
that is, $a+\sigma-2\geq \sigma$. So if \fbox{$ a\geq 2$} this condition is compatible.

We choose the less restrictive for each case, that is, if $\sigma \in (0,1]$ then $a\geq 2\sigma$ and if $\sigma \in (1,2)$ then $a\geq 2$. Summing up, we have

\begin{coro}\label{coroll4.7}
Under the assumptions of Theorem {\rm \ref{convergence}}, we also assume  that $\sigma \in (0,1]$ and
$$
D\Delta x^{\frac{a}{2}}\leq \Delta t\leq C(m,f) \Delta x^\sigma \quad \mbox{ for some} \ a\geq 2\sigma,
$$
or $\sigma \in (1,2)$ and
$$
D\Delta x^{a+\sigma-2}\leq \Delta t\leq C(m,f) \Delta x^\sigma
\quad \mbox{ for some}  \ a\geq 2\,,
$$
where  $D>0$ is a constant. Then,  $F_J=O(\Delta t+\Delta x^{2-\sigma}).$
\end{coro}

\section{Optimal rate of convergence}
Once convergence is established, we could  ask ourselves which values of the discretization order $a$ will produce the best rate of convergence with the least computational effort. The best we could expect is to have an error going to zero as $O(\Delta x^{2-\sigma})$, because even if we choose $\Delta t\sim \Delta x^\alpha$ with $\alpha < 2-\sigma$, there is an error term that goes like $O(\Delta x^{2-\sigma})$.

To start, we point out that when we apply the Corollary with $a=2\sigma$ and $D\Delta x^{\sigma}\leq \Delta t\leq C\Delta x^\sigma$,  we get the error estimate
\[
E_J=O(\Delta t+\Delta x^{2-\sigma})=O(\Delta x^\sigma+\Delta x^{2-\sigma}).
\]
This is an explicit convergence result, and in fact it is optimal for $\sigma =1$. In the other cases we can do better, as we show next.

\begin{teor}
Assume that $\sigma\in(0,1]$. Under the statements of Theorem \ref{convergence}, if $a=2(2-\sigma)$ then, for
\[D\Delta x^{2-\sigma}\leq\Delta t \leq C(m,f) \Delta x^{2-\sigma},\]
we have
\begin{equation}
F_J=O(\Delta x^{2-\sigma}).
\end{equation}
The same conclusion is reached if $\sigma\in(1,2)$ by choosing  $a=2$ and
\[
D\Delta x^{\sigma}\leq\Delta t \leq C(m,f) \Delta x^{\sigma}\,.
\]
\end{teor}

\noindent $\bullet$ If $\sigma\in (0,1]$, we argue as follows. We want to find values of $a$ that produce the optimal rate of convergence. It is clear that the best error we can get is $O(\Delta x^{2-\sigma})$. Since $D\Delta x^{\frac{a}{2}}\leq \Delta t$, and we want
\[
\Delta t\le K\Delta x^{2-\sigma},
\]
any value of $a\ge 2(2-\sigma)$  allows us to take $ \Delta t=\Delta x^{\frac{a}{2}}$
and get the conclusion. The minimal value is $a= 2(2-\sigma)$.
 Notice that, since $\sigma \in(0,1)$, then $\Delta x^{2-\sigma}\leq \Delta x^\sigma$

\noindent $\bullet$ If ${\sigma\in (1,2)}$, restriction $\Delta t\leq C(m,f)\Delta x^\sigma$ gives optimal regularity because $\sigma > 2-\sigma$ and so
\[\Delta t \sim \Delta x^\sigma \leq \Delta x^{2-\sigma},\]
and so, we have only to ask for $\Delta x^{a+\sigma-2}=\Delta x^{\sigma}$, that is, $a=2$.

\begin{figure}[h!]
	\begin{center}
		\includegraphics[width=0.8\textwidth]{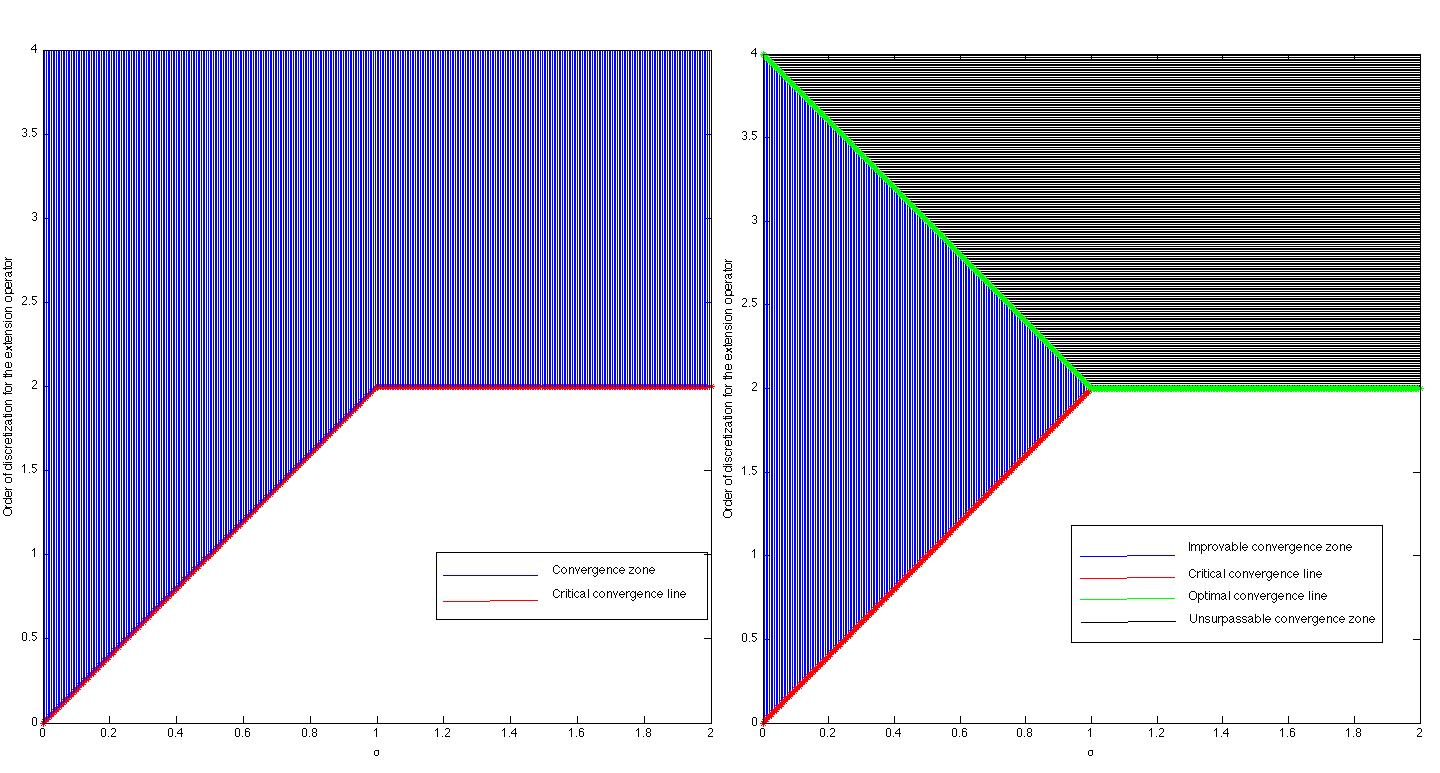}
		\caption{Convergence zones depending on the discretization of $L_\sigma$.}
        		\label{m10}
	\end{center}
\end{figure}


\subsection{Order of discretization of $L_\sigma$ for optimal rate of convergence}

Let us examine the required order of discretization of the extension operator to reach the optimal rate of convergence se have just studied. We recall that the operator $L_\sigma v=\nabla \cdot (y^{1-\sigma}\nabla v)$ can also be written as
\begin{equation}
L_\sigma v(x,y)= y^{1-\sigma }\Delta v(x,y)+(1-\sigma) y^{-\sigma} \frac{\partial v}{\partial y}(x,y).
\end{equation}

Let $\Delta^c$ be a discretization of order $c$ of the laplacian $\Delta$, that is,
\[\max_{i,k}|\Delta^c v_i^k-\Delta v(x_i,y_k)|=O(\Delta x^c).\]

Let also $D_y^d$ be a discretization of order $d$ of the first derivative, that is,
\[\max_{i,k}|D_y^d v_i^k- \frac{\partial v}{\partial y}(x_i,y_k)|=O(\Delta x^d).\]

Now, a natural way of the defining the discretization $L_\sigma^{c,d}$ of the operator $L_\sigma$ is
\begin{equation}\label{discoperator}
L_\sigma^{c,d} v_i^k= y^{1-\sigma }\Delta^c v_i^k +(1-\sigma) y^{-\sigma} D_y^d v_i^k,
\end{equation}
and then,
\[|L_\sigma^{c,d} v_i^k-L_\sigma v(x_i,y_k)|\leq y_k^{1-\sigma}|\Delta^c v_i^k-\Delta v(x_i,y_k)|+y_k^{-\sigma}|1-\sigma||D_y^d v_i^k- \frac{\partial v}{\partial y}(x_i,y_k)|\]
\[\leq  K y_k^{1-\sigma} \Delta x^{c}+K|1-\sigma|  y_k^{-\sigma} \Delta x^d. \ \ \ \]

\noindent $\bullet$ Case ${\sigma\in (0,1]}$. Then,
\[
y_k^{1-\sigma} \Delta x^{c}\leq  Y^{1-\sigma} \Delta x^{c},\]
and
\[
y_k^{-\sigma} \Delta x^d\leq \frac{1}{\Delta x^{\sigma}}\Delta x^d=\Delta x^{d-\sigma},
\]

Then, for any problem posed in a bounded domain ($Y<+\infty$), we have
\begin{equation}
\max_{i,k}|L_\sigma^{c,d} v_i^k-L_\sigma v(x_i,y_k)|=O(\Delta x^{c}+\Delta x^{d-\sigma}).
\end{equation}

As we have seen in the previous section, for optimal rate of convergence in this case $\sigma\in (0,1]$ we need a discretization of \ $L_\sigma$ \ of order greater than $a=2(2-\sigma)$, which implies
\begin{equation}
c=2(2-\sigma) \mbox{ and } d=4-\sigma.
\end{equation}

\noindent $\bullet$ Case ${\sigma\in (1,2)}$. Then,
\[
y_k^{1-\sigma} \Delta x^{c}\leq  \Delta x^{1-\sigma} \Delta x^{c}=\Delta x^{c+1-\sigma},\]
and
\[
y_k^{-\sigma} \Delta x^d\leq \Delta x^{d-\sigma}.
\]

Then, for any problem, we have
\begin{equation}
\max_{i,k}|L_\sigma^{c,d} v_i^k-L_\sigma v(x_i,y_k)|=O(\Delta x^{c+1-\sigma}+\Delta x^{d-\sigma}).
\end{equation}

In the case $\sigma \in (1,2)$, the optimal regularity is obtained with a discretization of $L_\sigma$ of order greater than or equal to 2, which implies
\begin{equation}
c=1+\sigma \mbox{ and } d=2+\sigma.
\end{equation}

\noindent {\bf Corollary.} {\sl Using discretizations of integer order for derivatives of integer order, we arrive at the following required order of discretization  depending on $\sigma$:

(i) If ${\sigma\in (0,1/2)}$, then $c=4$ and $d=4$.

(ii) If ${\sigma\in (1/2,1)}$, then $c=3$ and $d=4$.

(iii) If  ${\sigma=1}$, the we do not have the second term in $L_\sigma$ and so we only need $c=2$.

(iv) If ${\sigma\in (1,2)}$, then $c=3$ and $d=4$.}

\medskip

We observe that, except in the case $\sigma =1$, the computational cost of the proposed higher order discretization  can be rather expensive in order to obtain the optimal rate of convergence. Moreover, this higher order discretization has to be adapted to their position in the mesh. For example, if we want to have a third order discretization of $\frac{\partial}{\partial y}$ we will need a four points rule. This rule can not be the same if we are close to the boundaries $\Gamma_h$ or the $\Gamma_d$, because there will not be enough points in the direction of the boundary.

We think that some improvements in this section could be obtained. One possibility passes through obtaining direct discretizations of the extension operator $L_\sigma$. One may also think of using an adapted scheme where higher order discretizations are used only near the boundary. This kind of higher order discretizations are well presented by Ciarlet in \cite{ciarlet}. This reference also gives sufficient and necessary  conditions on the finite difference matrix  for the discretizations to be compatible with the discrete maximum principle.

Of course, we are not forced to look for the optimal rate of convergence in order to have the best numerical convergence. In this case, we can of course relax the requirements on the discretization order  of the extension operator.


\section{Lower order convergence}

We have been concerned with obtaining optimal rates of convergence. Frequently, this will imply higher order discretizations and consequently higher computational cost. Here we will study the minimal order of discretization needed to  have convergence, even if this convergence may be very slow. We recall that, according to  Theorem \ref{convergence} the error of the numerical method is estimated as
\[
E_J=O(\Delta t+\Delta x^{2-\sigma}+ \frac{\Delta x^a}{\Delta t} )\,,
\]
where $a$ is the order of discretization of the extension operator. We obtain convergence by requiring that either \
$\Delta x^a\leq \Delta t \Delta t^\epsilon$,  or \  $ \Delta x^a\leq \Delta t \Delta x^{\delta}$
for some $\epsilon,\delta >0$.

\medskip

\noindent {\sc Condition 1:} \ $\Delta x^{a}\leq \Delta t^{1+\epsilon}$, that is, $\Delta x^{\frac{a}{1+\epsilon}}\leq \Delta t$. We need this condition to be compatible with the CFL condition $\Delta t\leq \Delta x^\sigma$, that is,
\[
\Delta x^{\frac{a}{1+\epsilon}}\leq \Delta x^{\sigma}.
\]
So, in this case we need, {$a> \sigma$}.

\medskip

\noindent {\sc Condition 2:} \ $\Delta x^{a}\leq \Delta t\Delta x^{\delta}$, that is, $\Delta x^{a-\delta}\leq \Delta t$. We need this condition to be compatible with the CFL condition $\Delta t\leq \Delta x^\sigma$, that is,
\[\Delta x^{a-\delta}\leq \Delta x^{\sigma}.\]
So, we again need, {$a>\sigma$} (the precise condition is the same if $\delta=\epsilon\sigma$).

\begin{coro}
Under the statements of Theorem {\rm \ref{convergence}}, assume also that, for a fixed $\delta > 0$ and some constant $D>0$:
\begin{equation}
D\Delta x^{a-\delta }\leq \Delta t\leq C(m,f) \Delta x^\sigma \mbox{ for some } a\geq \sigma+\delta.
\end{equation}
Then,
\begin{equation}
F_J=O(\Delta t+\Delta x^{2-\sigma}+ \Delta x^{\delta}).
\end{equation}
\end{coro}

Using the notation of the discretization $(\ref{discoperator})$, the condition of the Corollary implies the following orders of discretization:

-If $\underline{\sigma\in (0,1)}$, \ $c=\sigma+\delta \mbox{ and } d=2\sigma+\delta.$

-If $\underline{\sigma\in [1,2)}$, \ $c=2\sigma-1+\delta \mbox{ and } d=2\sigma+\delta.$

\begin{coro}
As in the previous sections, we only use discretizations of integer order for derivatives of integer order. These are the orders of discretization that we need, depending on $\sigma$:

{\rm (i)} If ${\sigma\in (0,1/2)}$, then $c=1$ and $d=1$.

{\rm (ii)} If ${\sigma\in (1/2,1)}$, then $c=1$ and $d=2$.

{\rm (iii)} If ${\sigma=1}$, we do not have the second term, and so $c=2$.

{\rm (iv)} If ${\sigma\in (1,3/2)}$, then $c=2$ and $d=3$.

{\rm (v)} If  ${\sigma\in (3/2,2)}$, then $c=3$ and $d=4$.

\end{coro}

\section{The problem in the whole space}

We have been comparing solutions of the numerical scheme  (\ref{numericmethod}) with theoretical solutions of \eqref{problemextbdd}, with both problems  posed in a bounded domain. The proofs presented in this paper require a certain regularity of the theoretical solutions. This kind of results are already known for solutions of the problem posed in $\mathbb{R}^{N+1}_+$ with quite general data, (\cite{afracpor4}). Similar results are under study for bounded domains. In this section we propose an application of our previous results. Indeed,  we will compare the solution to the numerical scheme (\ref{numericmethod}) posed in the bounded domain $\Omega= [-X,X]\times [0, X]$ with the solution to problem in the whole  space. Note that technically the problem is posed in $\mathbb{R}^{N+1}_+$, cf. \eqref{problemext}. Therefore, the theoretical solution has the required regularity.

The comparison can only be done in the domain where the numerical scheme is defined. A difficulty appears with this kind of comparison, since from now on we assume that the theoretical solution $w\not = 0$ at $\Gamma_h$ in view of the property of strict positivity of all nonnegative solutions proved in \cite{afracpor, afracpor2}. This implies that an extra error will be introduced to the numerical solution, coming from the lateral boundary. Since we want a convergence result from the numerical solution in the bounded domain to the problem posed in $\mathbb{R}^{N+1}_+$, we will make $\Omega\to \mathbb{R}^{N+1}_+$ as $\Delta x\to 0$.



We need to control the error coming from the lateral boundary. In \cite{jlbar}, an upper bound for the solution with compactly supported initial data is found by passing through the Barenblatt solutions of problem (\ref{problem}). The upper bound is,
\[
u_M^*(x,t)=t^{-\alpha} F(|x|t^{-\beta}),
\]
where $F(\xi)\leq C |\xi|^{-(N+\sigma)}$ and
\[
\alpha=\frac{N}{N(m+1)+\sigma}, \ \ \ \ \beta=\frac{1}{N(m+1)+\sigma}.
\]
Since $-\alpha+\beta(N+\sigma)=\beta\sigma$, we have the next bound in $\Gamma_h$,
\[
u_M^*(X,t)\leq Ct^{-\alpha+\beta(N+\sigma)}\frac{1}{X^{N+\sigma}}\leq Ct^{\beta\sigma}\frac{1}{X^{N+\sigma}}\leq C\cdot T^{\beta\sigma}\frac{1}{X^{N+\sigma}}.
\]
Then, if we impose the following extra condition condition in the domain,

\begin{equation}\label{domaincond}
\frac{C T^{\beta\sigma}}{|X|^{N+\sigma}}\leq K \Delta x^a,
\end{equation}
for a fixed constant $K>0$,  we can adapt the proofs of Theorems \ref{localtruncerror} and \ref{convergence} to obtain the desired convergence. The easiest choice is $K=1$ and so, condition (\ref{domaincond}) becomes
\begin{equation}
|X|\geq\frac{ (C\cdot T^{\beta\sigma})^{\frac{1}{N+\sigma}}}{\Delta x^\frac{a}{N+\sigma}}.
\end{equation}

The rest of the changes are as follows: in \textbf{Theorem \ref{localtruncerror}}, the local truncation error in the interior nodes of $\Omega$  and in $\Gamma_d$ still being the same but is not zero anymore in $\Gamma_h$. Now if $(x_i,y_k) \in \Gamma_h$,
\[(\tau_j)_i^k=(w_j)_i^k\leq  \frac{C \cdot T^{\beta\sigma}}{|X|^{N+\sigma}}\leq\Delta x^a.\]
and so $\Lambda=O\,(\Delta t(\Delta x^{2-\sigma}+\Delta t)+\Delta x^a)$ as before.

In \textbf{Theorem \ref{convergence}}, again the only change is that the error in $\Gamma_h$ is not zero. But, if $(x_i,y_k) \in \Gamma_h$,
\[(e_j)_i^k=(w_j)_i^k-(W_j)_i^k=(w_j)_i^k\leq D\Delta x^a.\]
and so \[E_J=O(\Delta t+\Delta x^{2-\sigma}+\frac{\Delta x^a}{\Delta t}).\]

We thus get the following result.

\begin{teor}
Let $w$ be the solution to Problem (\ref{problemext}) (posed in $\RR$) and $(W_j)_i^k$  be the solution to system (\ref{numericmethod}) (posed in the bounded domain $\Omega=[-X,X]\times[0,X]$) with $m\geq 1$ and compactly supported initial data $f$. Assume that:

1. There exists a constant $C(m,f)>0$ such that
\[\Delta t\leq C(m,f)\Delta x^\sigma.\]

2. The boundary of the domain $\Omega=[-X,X]\times[0,X]$  is such, for some constant $L=L(m,T,f,N,\sigma)>0$ we have,
\[|X|\geq\frac{ L}{\Delta x^\frac{a}{N+\sigma}}.
\]
Then,
\[\max_{i,j,k}|w(x_i,y_k,t_j)-(W_j)_i^k|=O(\Delta t+\Delta x^{2-\sigma}+\frac{\Delta x^a}{\Delta t}).\]
\end{teor}
\begin{nota}
Condition 2 says that as ${\Delta x \to 0}$ we need $|X|\to \infty$ and so $\Omega \to  \RR$.
\end{nota}

\section{Extensions and comments}

\noindent $\bullet$ As a natural extension of the results of this paper we can consider the same equation with
data of any sign, and also the equation with exponent $0<m<1$ (the fast diffusion case). The method we use here does not directly apply to such cases. For instance, in the case of signed data, the solutions are not supposed to be classical, so a different approach is needed.
\normalcolor

\section*{Acknowledgments}
Both authors partially supported by the  Spanish Project MTM2011-24696.
The first author  also supported by a FPU grant from Ministerio de Educaci\'on, Ciencia y Deporte, Spain.



\vspace{.5cm}
\textbf{Keywords: }Nonlinear diffusion equation, fractional Laplacian, numerical method, finite difference, rate of convergence.


\noindent {\sc Authors' address: }  Departamento de Matem\'aticas, Universidad Aut\'onoma de Madrid,\\
Campus de Cantoblanco, 28049 Madrid, Spain.







\end{document}